\documentclass[reqno,11pt]{amsart}
\usepackage{amsmath,amssymb,amsthm,hyperref}

\newtheorem{thm}{Theorem}[section]
\newtheorem{cor}[thm]{Corollary}

\newtheorem{prop}[thm]{Proposition}
\newtheorem{lem}[thm]{Lemma}

\theoremstyle{definition}

\newtheorem{exa}[thm]{Example}

\theoremstyle{remark}

\renewcommand{\bar}{\overline}
\newcommand{\excise}[1]{}

\newcommand{\qandq}{\quad\text{and}\quad}

\newcommand{\0}{\emptyset}

\newcommand{\compn}{\vDash}

\newcommand{\Cc}{\mathbb{C}}
\newcommand{\Ff}{\mathbb{F}}
\newcommand{\Xx}{\mathbb{X}}

\newcommand{\Zz}{\mathbb{Z}}

\newcommand{\GG}{\mathcal{G}}
\newcommand{\OO}{\mathcal{O}}

\newcommand{\tp}{\otimes}

\newcommand{\sm}{\setminus}
\newcommand{\HH}{\mathcal{H}}
\newcommand{\D}{\Delta}
\DeclareMathOperator{\ps}{ps}
\DeclareMathOperator{\rk}{rk}
\DeclareMathOperator{\nul}{null}

\newcommand{\prin}{\Pi}

\newcommand{\Flat}{\mathcal{F}}  
\newcommand{\AO}{\mathcal{A}}  
\newcommand{\QSym}{\mathcal{Q}Sym}

\author{Brandon Humpert}
\address{Department of Mathematics, University of Kansas, 405 Snow Hall,
1460 Jayhawk Blvd, Lawrence, Kansas 66045-7594}
\author{Jeremy L.\ Martin}
\address{Department of Mathematics, University of Kansas, 405 Snow Hall,
1460 Jayhawk Blvd, Lawrence, Kansas 66045-7594}
\title[The Incidence Hopf Algebra of Graphs]{The Incidence Hopf Algebra of Graphs}
\thanks{An extended abstract of this article appeared in the
proceedings of the 23rd International Conference on Formal Power Series
and Algebraic Combinatorics \cite{ExtAbs}.}
\date{September 12, 2011}

\begin{document}
\begin{abstract}
The \emph{graph algebra} is a commutative, cocommutative, graded,
connected incidence Hopf algebra, whose basis elements correspond to
finite graphs and whose Hopf product and coproduct admit simple
combinatorial descriptions.  We give a new formula for the antipode in
the graph algebra in terms of acyclic orientations; our formula
contains many fewer terms than Takeuchi's and Schmitt's more general
formulas for the antipode in an incidence Hopf algebra.  Applications
include several formulas (some old and some new) for evaluations of
the Tutte polynomial.
\end{abstract}

\maketitle

\section{Introduction}

The \emph{graph algebra $\GG$} is a commutative, cocommutative, graded, connected
Hopf algebra, whose basis elements correspond to finite graphs,
and whose Hopf product and coproduct admit simple combinatorial
descriptions.  The graph algebra was first considered by Schmitt in
the context of incidence Hopf algebras \cite[\S12]{Schmitt} and
furnishes an important example in the work of Aguiar, Bergeron and
Sottile \cite[Example~4.5]{ABS}.

In this paper, we derive a nonrecursive formula
(Theorem~\ref{antipode-theorem}) for the Hopf antipode in $\GG$.  Our
formula is specific to the graph algebra in that it involves acyclic
orientations.  Therefore, it is not merely a specialization of
the antipode formulas of Takeuchi \cite{Takeuchi} or Schmitt \cite{Schmitt}
(in the more general settings of, respectively,
connected bialgebras and incidence Hopf algebras).
Aguiar and Ardila \cite{AA-PC} have independently discovered a more general
antipode formula than ours, in the context of Hopf monoids; their work will
appear in a forthcoming paper.

Our formula turns out to be well suited for studying polynomial graph
invariants, including the Tutte polynomial $T_G(x,y)$ (see
\cite{BryOx}) and various specializations of it.  Specifically, to
every graph $G$ and character $\zeta$ on the graph algebra, we
associate the function $P_{\zeta,G}(k)$ whose value at an integer $k$
is $\zeta^k(G)$, where the superscript denotes convolution power.  For
example, if $\zeta$ is the characteristic function of edgeless graphs,
then $P_{\zeta,G}(k)$ is the chromatic polynomial of $G$.  In fact, it
turns out that $P_{\zeta,G}(k)$ is a polynomial function of $k$ for
\emph{all} characters~$\zeta$, so we may regard $P_\zeta$ as a map
$\GG\to\Cc[k]$ sending $G$ to $P_{\zeta,G}(k)$.  This map is in fact a
morphism of Hopf algebras; it is not graded (so it is not quite a
morphism of \emph{combinatorial} Hopf algebras in the sense of Aguiar,
Bergeron and Sottile~\cite{ABS}) but does preserve the canonical
filtration by degree.  Together with the antipode formula, this
observation leads to combinatorial interpretations of the convolution
inverses of several natural characters, as we discuss in
Section~\ref{char-section}.

The Tutte polynomial $T_G(x,y)$ can itself be viewed as a character on
the graph algebra.  We prove that its $k$-th convolution power itself
is a Tutte evaluation at rational functions in $x,y,k$
(Theorem~\ref{tutte-char}).  This result implies several well-known
formulas such as Stanley's formula for acyclic orientations in terms
of the chromatic polynomial~\cite{AO}, as well as some interpretations
of less familiar specializations of the Tutte polynomial, and an
unusual-looking reciprocity relation between complete graphs of
different sizes (Proposition~\ref{wacko-reciprocity} and Corollary~\ref{funky}).

The authors thank Marcelo Aguiar, Federico Ardila, Diego Cifuentes,
Aaron Lauve, and Vic Reiner for numerous helpful conversations, and
two anonymous referees for valuable assistance in improving the
exposition.

\section{Hopf algebras}
\subsection{Basic definitions}
We briefly review the basic facts about Hopf algebras, omitting most of the
proofs.  Good sources for the full details include Sweedler
\cite{Sweedler} and (for combinatorial Hopf algebras) Aguiar, Bergeron
and Sottile \cite{ABS}.  For the more general setting of Hopf monoids,
see Aguiar and Mahajan \cite{AgMa}.  We do not know of specific
references for Lemma~\ref{prin-morphism} and
Proposition~\ref{to-binomial-alg}, but they are well known as part of
the general folklore of (combinatorial) Hopf algebras.

Fix a field $\Ff$ of characteristic 0 (typically $\Ff=\Cc$). A \emph{bialgebra} $\HH$ is a
vector space over $\Ff$ equipped with linear maps
$$
m: \HH \tp \HH \to \HH, \qquad u: \Ff \to \HH, \qquad \D: \HH \to
\HH \tp \HH,\qquad \epsilon: \HH \to \Ff,
$$
respectively the \emph{multiplication}, \emph{unit}, \emph{comultiplication}, and
\emph{counit}, such that the following properties are
satisfied:
\begin{enumerate}
\item  $m \circ (m \tp I) = m \circ (I \tp m)$ (associativity);
\item $m \circ (u \tp I) = m \circ (I \tp u) = I$ (where $I$ is the identity map on $\HH$);
\item $(\D \tp I) \circ \D = (I \tp \D) \circ \D$
(coassociativity);
\item $(\epsilon \tp I) \circ \D = (I \tp \epsilon) \circ \D = I$; and
\item $\D$ and $\epsilon$ are multiplicative (equivalently, $m$ and $u$ are comultiplicative).
\end{enumerate}
If there exists a bialgebra automorphism $S: \HH \to \HH$
such that $m \circ (S \tp I) \circ \D = m \circ (I \tp S) \circ \D = u
\circ \epsilon$, then $\HH$ is a \emph{Hopf algebra} and
$S$ is its \emph{antipode}.  It can be shown that $S$ is the
unique automorphism of $\HH$ with this property.

It is often convenient to write expressions such as coproducts in \emph{Sweedler notation}, where the index of summation is suppressed: for instance, $\Delta(h)=\sum h_1\otimes h_2$
rather than $\Delta(h)=\sum_i h^{(i)}_1\otimes h^{(i)}_2$.

A Hopf algebra $\HH$ is \emph{graded} if $\HH = \bigoplus_{n\geq 0}
\HH_n$ as vector spaces, and multiplication and comultiplication
respect this decomposition, i.e.,
$$
m(\HH_i\otimes\HH_j)\subseteq\HH_{i+j}
\quad\text{and}\quad
\D(\HH_n)\subseteq\sum_{i+j=n}\HH_i\otimes\HH_j.
$$
If $h\in\HH_i$, we say that $h$ is \emph{homogeneous of degree~$i$}.
The algebra $\HH$ is \emph{connected} if $\dim(\HH_0) = 1$.
Most Hopf algebras arising naturally in
combinatorics are
graded and connected, and every algebra we consider henceforth will
be assumed to have these properties.

Let $\HH$ be a graded and connected
bialgebra.  There is a unique Hopf antipode on $\HH$,
defined inductively by the formulas
\begin{subequations}
\begin{eqnarray}
S(h)=h && \text{ for } h\in\HH_0, \label{antipode-zero}\\
(m \circ (I \tp S) \circ \D)(h) = 0 && \text{ for }
h\in\HH_i,\ i>0. \label{antipode-pos}
\end{eqnarray}
\end{subequations}
Formula \eqref{antipode-pos} can be rewritten more explicitly using Sweedler notation.
If $\Delta(h)=\sum h_1\otimes h_2$, then $\sum h_1S(h_2) = 0$, so solving for $S(h)$
gives
\begin{equation} \label{antipode-formula}
S(h)=-\sum h_1S(h_2),
\end{equation}
the sum over all summands in which the degree of $h_2$ is strictly less
than that of $h$.

A \emph{character} of a Hopf algebra $\HH$ is a multiplicative linear
map $\phi: \HH \to \Ff$.  The \emph{convolution product} of
two characters is $\phi\ast\psi = (\phi \tp \psi) \circ \D$.
That is, if $\Delta h=\sum h_1\tp h_2$, then
$$(\phi*\psi)(h)=\sum \phi(h_1)\psi(h_2)$$
with both sums in Sweedler notation.
We write $\phi^k$ for the $k$-th convolution power
of $\phi$; if $k<0$ then $\phi^k=(\phi^{-1})^{-k}$.
Convolution makes the set of characters
$\Xx(\HH)$ into a group, with identity $\epsilon$ and inverse given by
\begin{equation}
\label{character-antipode}
\phi^{-1} = \phi \circ S.
\end{equation}
There is a natural involutive
automorphism $\phi\mapsto\bar\phi$ of $\Xx(\HH)$, given by
$\bar\phi(h) = (-1)^n \phi(h)$ for $h \in \HH_n$.
If $\HH$ is a graded connected Hopf algebra and $\zeta\in\Xx(H)$, then the pair $(\HH, \zeta)$ is called a \emph{combinatorial Hopf algebra}, or CHA for short.
A \emph{morphism} of CHAs $\Phi:(\HH,\zeta)\to(\HH',\zeta')$ is a linear transformation $\HH\to\HH'$ that is a morphism of
Hopf algebras (i.e., a linear transformation that preserves the operations of a bialgebra) such that $\zeta\circ\Phi=\zeta'$.

\subsection{The binomial Hopf algebra}

The \emph{binomial Hopf algebra} is the ring of polynomials $\Ff[k]$
in one variable $k$, with the usual multiplicative structure;
comultiplication defined by $\D(f(k)) = f(k \tp 1 + 1 \tp k)$ and
$\Delta(1)=1\tp1$; counit $\epsilon(f(k)) = \epsilon_0(f(k))=f(0)$;
and character $\epsilon_1(f(k)) = f(1)$.  A theme of this article,
that polynomial invariants of elements of a Hopf algebra $\HH$ can be
viewed as the values of a morphism $\HH\to\Ff[k]$.
The main result in this vein, Proposition~\ref{to-binomial-alg},
can be proved with elementary methods, but we instead
give a longer proof that illustrates the connection
to the work of Aguiar, Bergeron, and Sottile~\cite{ABS}.
In order to do so, 
we begin by reviewing some facts about compositions and quasisymmetric
functions; for more details, see, e.g., \cite[\S7.19]{EC2}.

Let $n$ be a nonnegative integer. A \emph{composition
of~$n$} is an ordered sequence $\alpha=(\alpha_1,\dots,\alpha_\ell)$
of positive integers such that $\alpha_1+\dots+\alpha_\ell=n$; in this
case we write $\alpha\compn n$.  The number $\ell=\ell(\alpha)$ is the
\emph{length} of $\alpha$.  The corresponding \emph{monomial
quasisymmetric function} is the formal power series
  \begin{equation} \label{defM}
  M_\alpha=\sum_{0<i_1<\cdots<i_\ell}x_{i_1}^{\alpha_1}\cdots x_{i_\ell}^{\alpha_\ell}
  \end{equation}
in countably infinitely many commuting variables $\{x_1,x_2,\dots\}$.
The $\Ff$-vector space spanned by the $M_\alpha$ is denoted $\QSym$.
This is in fact a Hopf algebra, with the natural addition,
multiplication, and unit; counit
$$\epsilon(M_\alpha)=\begin{cases} 1 & \text{ if } \ell(\alpha)=0,\\
0 & \text{ if } \ell(\alpha)>0;\end{cases}$$
and comultiplication
$$\Delta(M_{(\alpha_1,\dots,\alpha_\ell)}) = \sum_{i=0}^\ell M_{(\alpha_1,\dots,\alpha_i)}M_{(\alpha_{i+1},\dots,\alpha_\ell)}.$$
For $F(x_1,x_2,\dots)\in\QSym$, let
$\zeta_Q(F)$ be the number obtained by
substituting $x_1=1$ and $x_2=x_3=\cdots=0$. The map~$\zeta_Q$ is a
character on~$\QSym$.  Aguiar, Bergeron, and Sottile
\cite[Thm.~4.1]{ABS} proved that $(\QSym,\zeta_Q)$ is a terminal
object in the category of CHAs, i.e., that every CHA $(\HH,\zeta)$ has a
unique morphism
  $$\Psi:(\HH,\zeta)\to(\QSym,\zeta_Q),$$
given explicitly on $h\in\HH_n$ by
  $$\Psi(h)=\sum_{\alpha\compn n} \zeta_\alpha(h)M_\alpha\,;$$
here $\zeta_\alpha:\HH\to\Ff$ is the composite function
  $$\HH\xrightarrow{\Delta^{\ell-1}}\HH^{\otimes\ell}\xrightarrow{\pi_\alpha}\HH_{\alpha_1}\otimes\cdots\otimes\HH_{\alpha_\ell}\xrightarrow{\zeta^{\otimes\ell}}\Ff$$
where $\ell=\ell(\alpha)$ is the number of parts of $\alpha$,
and $\pi_\alpha$ is the tensor product of the canonical projections of $\HH$
onto the graded pieces $\HH_{\alpha_i}$.

For $F(x_1,x_2,\dots)\in\QSym$, let $\ps_k^1(F)$ be the number
obtained by substituting $x_1=\cdots=x_k=1$ and $x_{k+1}=\cdots=0$.
In particular, $\ps_1^1=\zeta_Q$.  The map $\ps_k^1$ is a
specialization of a map called the \emph{principal
specialization}~\cite[pp.~302--303]{EC2}. By \eqref{defM}, we have
$$\ps_k^1(M_\alpha)=\frac{k(k-1)\cdots(k-\ell(\alpha)+1)}{\ell(\alpha)!}=\binom{k}{\ell(\alpha)}.$$
Accordingly, we can regard $\ps_k^1$ as a map
  $$\prin:\QSym\to\Ff[k]$$
sending $M_\alpha$ to $\ps_k^1(M_\alpha)$.  (The reason for the apparently
redundant notation is that when we write $\ps_k^1$, we are regarding~$k$
as an integer, while when we write $\prin$, we are regarding~$k$ as the
indeterminate in the polynomial ring $\Ff[k]$.)

\begin{lem} \label{prin-morphism}
The map $\prin:\QSym\to\Ff[k]$ is a morphism of Hopf algebras.
Moreover, $\zeta_Q=\epsilon_1\circ\prin$.
\end{lem}

We remark that $\prin$ is not a morphism of \emph{combinatorial} Hopf
algebras because it is not graded (i.e., $\prin(M_\alpha)$
is not homogeneous), merely filtered by degree.

\begin{proof}
The definition of $\ps_k^1$ implies that $\prin$ is a homomorphism of
$\Ff$-algebras.  To see that it is in fact a Hopf morphism, we must
show that $(\prin\otimes\prin)\circ\Delta=\Delta\circ\prin$.
It suffices to check this
for the basis $\{M_\alpha\}$.  Let $x=k\otimes 1$ and $y=1\otimes k$; then
\begin{align*}
(\prin\otimes\prin)(\Delta M_\alpha)
&= (\prin\otimes\prin)\left(\sum_{j=0}^\ell M_{(\alpha_1,\dots,\alpha_j)}\otimes M_{(\alpha_{j+1},\dots,\alpha_\ell)}\right)\\
&= \sum_{j=0}^\ell \binom{x}{j} \binom{y}{\ell-j}
=\binom{x+y}{\ell}
= \Delta\binom{k}{\ell} = \Delta(\prin(M_\alpha)).
\end{align*}
(The third equality is a standard identity of binomial coefficients~\cite[Ex.~1.1.17]{EC1} that holds for all nonnegative integers $x,y$; therefore, it is an identity of polynomials.)

For the second assertion of the lemma, we have
$$\zeta_Q(M_\alpha)=\begin{cases}
1&\text{ if } \ell(\alpha)\leq 1\\
0&\text{ if } \ell(\alpha)>1 \end{cases}
~=~ \frac{k(k-1)\cdots(k-\ell+1)}{\ell!}\Big\vert_{k=1}
~=~ \epsilon_1(\prin(M_\alpha)).
$$
\end{proof}

We now come to the main result of this section.  Again, this fact is
not new, but is part of the folklore of (combinatorial) Hopf algebras.

\begin{prop}[Polynomiality]
\label{to-binomial-alg}
For every combinatorial Hopf algebra $(\HH,\zeta)$,
there is a CHA morphism
$$P_\zeta:(\HH,\zeta)\to(\Ff[k],\epsilon_1)$$
mapping $h$ to the unique polynomial $P_{\zeta,h}(k)$ such that
  $$P_{\zeta,h}(k)=\zeta^k(h) \qquad\forall k\in\Zz.$$
Moreover, if $h\in\HH_n$, then
$P_{\zeta,h}(k)$ is a polynomial in $k$ of degree at most $n$.
\end{prop}

\begin{proof}
We will show that $P_{\zeta,h}(k)=\prin(\Psi(h))$
for all $h\in\HH$.  It is not hard to see that $P_\zeta$
is a vector space homomorphism, so
it is sufficient to consider the
case that $h$ is homogeneous of degree~$n$.  The desired
equality follows from the calculation
\begin{subequations}
\begin{align}
P_{\zeta,h}(k) ~=~ \zeta^k(h) ~&=~ \sum \zeta(h_1)\cdots\zeta(h_k)\label{poly:a}\\
~&=~ \sum_{\alpha\compn n} \binom{k}{\ell(\alpha)} \zeta_\alpha(h)\label{poly:b}\\
~&=~ \prin \sum_{\alpha\compn n} \zeta_\alpha(h) M_\alpha\notag
~=~ \prin(\Psi(h)).
\end{align}
\end{subequations}
The sum in \eqref{poly:a} is Sweedler notation.  The only tricky
equality is \eqref{poly:b}; for this, note that each summand
$\zeta(h_1)\cdots\zeta(h_k)$ in \eqref{poly:a} arises from an ordered
list $(h_1,\dots,h_k)$ of homogeneous elements of $\HH$ whose degrees
sum to $n$.  Define the \emph{essence} of a summand
$\zeta(h_1)\cdots\zeta(h_k)$ to be the sublist of $(h_1,\dots,h_k)$
consisting of
elements of strictly positive degree.  
Each equivalence class of summands
with the same essence $(h_{i_1},\dots,h_{i_\ell})$
contains precisely $\binom{k}{\ell}$ summands
(since by the counit property, the positive-degree factors
may occur in any positions) and thus
contributes $\binom{k}{\ell}\zeta(h_{i_1})\cdots\zeta(h_{i_\ell})$
to the sum.  Collecting together all equivalence classes
whose essences have the same degree sequence $\alpha$ contributes
$\binom{k}{\ell(\alpha)}\zeta_\alpha(h)$.

Finally, observe that $\binom{k}{\ell}=\frac{k(k-1)\cdots
(k-\ell+1)}{\ell(\ell-1)\cdots 1}$ is a polynomial in~$k$ of degree~$\ell$,
and that every composition $\alpha\compn n$ has at most $n$ parts, so
$P_{\zeta,h}(k)$ is a polynomial in $k$ of degree at most $n$.
\end{proof}

One can also prove Proposition~\ref{to-binomial-alg} by direct
calculation, for instance, by
showing that $D^{n+1}P_{\zeta,h}(k)=0$, where $D$ is the difference
operator $DP(k)=P(k)-P(k-1)$.

Proposition~\ref{to-binomial-alg} 
provides a way of translating
characters on a Hopf algebra into polynomial invariants
of its elements, just as the Aguiar--Bergeron--Sottile theorem
translates characters into quasisymmetric-function invariants.
Passing from quasisymmetric functions to polynomials
may lose information, but may also lead to more explicit formulas.

\subsection{Graphs and the graph Hopf algebra}
We now describe the Hopf algebra that is the subject of this article.
(The literature contains many other instances of Hopf algebras of graphs;
for example, this is not the same Hopf structure as the algebra
studied by Novelli, Thibon and Thi\'ery~\cite{NTT}.)

First, we set up graph-theoretic notation and terminology.  The
notation $G=(V,E)$ means that $G$ is a finite, undirected
graph with vertex set $V$ and edge set $E$; we may then write
$G_{V',E'}$ for the subgraph with vertex set $V'$ and edge set $E'$.
(We could also write simply $(V',E')$, but we often wish to emphasize
that this graph is a subgraph of $G$.)  Loops and multiple edges are allowed.
The sets of vertices and edges of a
graph $G$ will be denoted $V(G)$ and $E(G)$ respectively; no confusion
should arise from this apparent abuse of notation.  
The numbers of vertices, edges and connected components are denoted
$n(G)$, $e(G)$, $c(G)$ respectively (or sometimes $n,e,c$).
The induced
subgraph on a vertex set $T\subseteq V$ will be denoted $G|_T$.
The complement of $T$ will be denoted $\bar T$.  If $S$ and $T$
are vertex sets, then $[S,T]$ denotes the set of all edges with one endpoint
in $S$ and one endpoint in $T$.
The complete graph on $n$
vertices is written $K_n$; note that we permit the possibility $n=0$.

The \emph{rank} $\rk(F)$ of a subset $F\subseteq E(G)$
is the size of any maximal acyclic subset of~$F$.  Meanwhile,
the set $F$ is called a \emph{flat} if, whenever the endpoints of an
edge~$e$ are connected by a path in $F$, then $e\in F$.
(These are precisely the flats of the graphic matroid of $G$.)
Equivalently, $F$ is a flat iff $\rk(F')>\rk(F)$ for every $F'\supsetneq F$.

For an edge $e\in E$, the \emph{contraction} $G/e$ is obtained by
identifying the two endpoints of $e$ (which is a trivial step if $e$
is a loop) and then removing~$e$.  For an edge set $F\subseteq E$, the
symbol $G/F$ denotes the graph obtained by successively
contracting every edge of $F$ (the order does not matter).  Observe
that if $F$ is a flat, then $G/F$ contains no loops.

An \emph{acyclic orientation} of $G$ is a choice of orientation of all
the edges that admits no directed cycles.  Let
\begin{align*}
\Flat(G) ~&=~ \{\text{flats of } G\},\\
\AO(G) ~&=~ \{\text{acyclic orientations of } G\},\\
a(G) ~&=~ |\AO(G)|.
\end{align*}
Note that if $G$ has one or more loops, then $a(G)=0$;
otherwise, the number of acyclic orientations is unchanged
upon replacing $G$ with its underlying simple graph.

Now we can define our central object of study.
The \emph{graph algebra} is the $\Ff$-vector space $\GG=\bigoplus_{n\geq0}\GG_n$,
where $\GG_n$ is the linear span of isomorphism classes of graphs on~$n$ vertices.
This is a graded connected Hopf algebra, with multiplication
$m(G \otimes H) = G \cdot H = G \uplus H$ (where $\uplus$ denotes disjoint union);
unit $u(1) = K_0$;
comultiplication
$$\Delta(G) = \sum_{T \subseteq V(G)} G|_T \otimes G|_{\bar T},$$
and counit
$$\epsilon(G) =
\begin{cases}
1 & \text{ if } n(G)=0,\\
0 & \text{ if } n(G)>0.
\end{cases}$$
The graph algebra is commutative and cocommutative; in particular,
its character group $\Xx(G)$ is abelian.
As proved by Schmitt~\cite[eq.~(12.1)]{Schmitt}, the antipode in $\GG$ is given
combinatorially by
$$S(G) = \sum_{\pi} (-1)^{|\pi|} |\pi|!\, G_\pi$$
where the sum runs over all ordered partitions $\pi$ of $V(G)$ into nonempty sets (or
``blocks''), and $G_\pi$ is the disjoint union of the
induced subgraphs on the blocks. 
This is a consequence
of Takeuchi's more general formula for connected Hopf algebras \cite[Lemma~14]{Takeuchi};
see also \cite[\S2.3.3 and \S8.4]{AgMa}, \cite[\S5]{MR2103213}, \cite{Montgomery}.

The graph algebra admits two canonical involutions on characters:
$$
\bar\phi(G)=(-1)^{n(G)}\phi(G),\qquad
\tilde\phi(G)=(-1)^{\rk(G)}\phi(G),
$$
where $\rk(G)$ denotes the graph rank of $G$ (that is, the number
of edges in a spanning tree).
As always, $\phi\mapsto\bar\phi$ is an automorphism of $\Xx(G)$; on the other hand,
$\phi\mapsto\tilde\phi$ is not.
The graph algebra was studied by Schmitt \cite{Schmitt}
and appears as the \emph{chromatic algebra} in
the work of Aguiar, Bergeron and Sottile \cite{ABS}, 
where it is equipped with the character
\begin{equation} \label{zeta-eqn}
\zeta(G) = \begin{cases}
1 & \text{ if $G$ has no edges,}\\
0 & \text{ if $G$ has at least one edge.}
\end{cases}
\end{equation}
We will study several characters on $\GG$ other than $\zeta$.

\section{A combinatorial antipode formula}

In this section, we prove a new combinatorial formula for the Hopf
antipode in $\GG$.  Unlike Takeuchi's and Schmitt's formulas, our
formula applies only to $\GG$ and and does not generalize to other
incidence algebras.  On the other hand, our formula involves many
fewer summands, which makes it useful for enumerative formulas
involving characters.  As noted in the introduction, Aguiar and Ardila
have independently discovered a more general antipode formula in the
context of Hopf monoids.


\begin{thm} \label{antipode-theorem}
Let $G=(V,E)$ be a graph with $n=|V|$.  Then
$$S(G) = \sum_{F\in\Flat(G)} (-1)^{n-\rk(F)} a(G/F) G_{V,F}.$$
\end{thm}

\begin{proof}
We proceed by induction on $n$. If $G$ has no vertices, i.e., $G=K_0$, then $S(K_0)=K_0$ by
\eqref{antipode-zero}.  Indeed, $\Flat(K_0)=\{\0\}$, so the desired formula reduces to $S(K_0) = K_0$.

On the other hand, if $G$ has at least one vertex, then by  \eqref{antipode-formula} we have
\begin{align}
S(G)
&=~ -\sum_{\0 \neq T \subseteq V} G|_T \cdot S(G|_{\bar{T}}) \notag\\
&=~ -\sum_{\0 \neq T \subseteq V} G|_T \ \sum_{F\in\Flat(G|_{\bar{T}})} (-1)^{n-|T|-\rk(F)} a(G|_{\bar{T}}/F) G_{\bar{T}, F} \notag\\
&=~ -\sum_{\0 \neq T \subseteq V} G|_T \ \sum_{F\in\Flat(G|_{\bar{T}})} \ \ \sum_{\OO\in\AO(G|_{\bar{T}}/F)} (-1)^{n-|T|-\rk(F)} G_{\bar{T},F}.\label{before-bijection}
\end{align}

Now we establish a bijection which will allow us to interchange the order of summation.

First, suppose we are given a nonempty vertex set~$T\subseteq V$, a flat~$F$
of~$G|_{\bar{T}}$, and an acyclic orientation~$\OO$
of~$G|_{\bar{T}}/F$.  Let~$F' = E(G|_T) \cup F$; this is
a flat of~$G$.  Moreover, we can construct an
acyclic orientation~$\OO'$ of~$G/F'$ by orienting all edges
in~$[\bar{T}, \bar{T}]$ as in~$\OO$, and orienting all edges in~$[T,
  \bar{T}]$ towards~$\bar{T}$.  Let~$S_{\OO'}$ be the set of sources
of $\OO'$ (that is, vertices with no in-edges); then the image~$T'$
of~$T$ under the contraction of~$F'$ is a nonempty subset
of~$S_{\OO'}$.

Second, suppose we are given a flat~$F'$ of~$G$, an acyclic orientation~$\OO'$
of~$G/F'$, and a set~$T'$ such that $\0 \neq T' \subseteq S_{\OO'}$.
Let~$T$ be the inverse image of~$T'$ under contraction of~$F'$. Then
$F=F'\sm E(G|_T)$ is a flat of~$G|_{\bar{T}}$, and we can
construct an acyclic orientation~$\OO$
of~$G|_{\bar{T}}/F$ by orienting all edges as in~$\OO'$.

It is straightforward to check that these constructions
are inverses.  Therefore, we have a bijection
$$
\left\{(T,F,\OO) ~\Big\vert~ \begin{array}{l}\0\neq T\subseteq V(G)\\ F\in\Flat(G|_{\bar T})\\ \OO\in\AO(G|_{\bar T}/F)\end{array}\right\}
\to
\left\{(F',\OO',T') ~\Big\vert~ \begin{array}{l}F'\in\Flat(G)\\ \OO'\in\AO(G/F')\\
\0\neq T'\subseteq S_{\OO'}
\end{array}\right\}
$$
with the following properties:
\begin{itemize}
\item $|T'|$ is the number of components of~$G|_T$;
\item $|T|-|T'|=\rk(G|_T)=\rk(F') - \rk(F)$, so $|T| + \rk(F) = |T'| + \rk(F')$;
\item $G|_T \cdot G_{\bar{T}, F} = G_{V, F'}$ in the graph algebra $\GG$.
\end{itemize}
Therefore, \eqref{before-bijection} gives
\begin{align*}
S(G)~
&=~ -\sum_{F'\in\Flat(G)} \ \ \sum_{\OO'\in\AO(G/F')} \ \ \sum_{\0 \neq T' \subseteq S_{\OO'}} (-1)^{n - |T'| - \rk(F')} G_{V, F'} \\
&=~ -\sum_{F'\in\Flat(G)} (-1)^{n - \rk(F')} G_{V, F'} \sum_{\OO'\in\AO(G/F')} \ \ \sum_{\0 \neq T' \subseteq S_{\OO'}} (-1)^{|T'|} \\
&=~ \sum_{F'\in\Flat(G)} (-1)^{n - \rk(F')} a(G/F') G_{V, F'}.\qedhere
\end{align*}
\end{proof}

\subsection{Inversion of characters}
\label{char-section}
We now apply the antipode formula to give combinatorial interpretations of
several instances of inversion in the group of characters.

\begin{prop}
\label{char-inversion}
Let $\Omega$ be any family of graphs such that 
$G\uplus H\in \Omega$ if and only if $G\in \Omega$ and $H\in \Omega$;
equivalently, such that the function
$$\psi_\Omega(G) =
\begin{cases}
1 & \text{ if } G\in \Omega,\\
0 & \text{ if } G\not\in \Omega
\end{cases}$$
is a character.  Then
$$\psi_\Omega^{-1}(G) = \sum_{F\in\Flat(G):\ G_{V,F}\in \Omega} (-1)^{n-\rk(F)} a(G/F).$$
\end{prop}

\begin{proof}
From equation~\eqref{character-antipode} and Theorem~\ref{antipode-theorem},
we have
\begin{align*}
\psi_\Omega^{-1}(G) &= \psi(SG)
= \sum_{F} (-1)^{n - \rk(F)} a(G/F) \psi(G_{V, F})\\
&= \sum_{F\in \Omega} (-1)^{n - \rk(F)} a(G/F).\qedhere
\end{align*}
\end{proof}

\begin{exa}
Let $\Omega$ be the family of graphs with no edges.  Then $\psi_\Omega$
is just the character~$\zeta$ of~\eqref{zeta-eqn}, and
$P_{\zeta,G}(k)$ is the chromatic polynomial $\chi(G;k)$ of~$G$.  Therefore,
Proposition~\ref{char-inversion} implies that
$\psi_\Omega^{-1}(G)=\chi(G;-1)= (-1)^n a(G)$, a classic theorem of
Stanley~\cite{AO}.
\end{exa}

\begin{exa}
Let $\Omega$ be the family of acyclic graphs, and let $\alpha=\psi_\Omega$. Then
$$\alpha^{-1}(G) = \sum_{\text{acyclic flats $F$}} (-1)^{n-\rk(F)} a(G/F).$$

We examine two special cases.  First, suppose that $G=C_n$, the cycle of length~$n$.  The acyclic flats of $G$ are just the sets
of $n-2$ or fewer edges, so an elementary calculation (which we omit)
gives $\alpha^{-1}(C_n)= (-1)^n+1$, the Euler characteristic of an $n$-sphere.

For many other families~$\Omega$, the $\Omega$-free flats of $C_n$ are just its
flats, i.e., the edge sets of cardinality $\neq n-1$.  In such cases,
the same omitted calculation gives $\psi_\Omega^{-1}(C_n)=(-1)^n$.

Second, suppose that $G=K_n$.  Now the acyclic flats of $G$ are matchings, i.e.,
sets of edges that cover no vertex more than once.  For
$0\leq m\leq\lfloor n/2\rfloor$, the number of $m$-edge matchings is
$n!/(2^m(n-2m)!m!)$, and contracting each such matching yields a graph
whose underlying simple graph is $K_{n-m}$.  Therefore
$$\alpha^{-1}(K_n) = \sum_{m=0}^{\lfloor n/2 \rfloor} (-1)^{n-m}
\frac{n!}{2^m(n-2m)!m!} (n-m)!.$$
Starting at $n=1$, these numbers are as follows:
$$-1,\ 1,\ 0,\ -6,\ 30,\ -90,\ 0,\ 2520,\ -22680,\ 113400,\ 0,\ -7484400,\ \dots.$$
This is sequence~\href{http://oeis.org/A009775}{A009775} in \cite{EIS}, for which the exponential generating
function is $-\tanh(\ln(1+x))$.
\end{exa}

\begin{exa} \label{avoidance-character}
Fix any connected graph $H$.  Say that $G$ is \emph{$H$-free} if it has no
subgraph isomorphic to $H$.  (This is a stronger condition than saying that $G$ has no
\emph{induced} subgraph isomorphic to $H$.)
The corresponding \emph{avoidance character} $\eta_H$ is defined by
$$
\eta_H(G) = \begin{cases}
1 & \text{ if $G$ is $H$-free},\\
0 & \text{ otherwise}.
\end{cases}$$
Avoidance characters are special cases of the characters described by Proposition~\ref{char-inversion};
specifically, $\eta_H=\psi_\Omega$, where $\Omega$ is the family of graphs with no subgraph isomorphic
to~$H$.
For instance, $\eta_{K_1}=\epsilon$ and $\eta_{K_2}=\zeta$;
more generally, if $H=K_{m,1}$ is the complete bipartite graph with
partite sets of sizes $m$ and~1, then the corresponding avoidance character $\eta_H$
detects whether or not $G$ has maximum degree strictly less than~$m$.
In general, for an avoidance character $\eta_H$, the summands in Proposition~\ref{char-inversion}
include only the $H$-free flats.

In the case that $T$ is a tree, every subset $F\subseteq E(T)$
is a flat, and every contraction
$T/F$ is acyclic, so all $2^{e(T/F)}$ orientations of~$T/F$ are acyclic.
Therefore, Proposition~\ref{char-inversion} simplifies to
  $$\eta_H^{-1}(T) ~=~ \sum_F (-1)^{r+1-|F|} 2^{r-|F|}
  ~=~ -\sum_F (-2)^{r-|F|}$$
where $r=r(T)=n(T)-1$, and both sums run over all $H$-free forests
$F\subseteq T$.
\end{exa}

\begin{exa}
For every avoidance character $\eta_H$,
the polynomial $P_{\eta_H}(G;k)$ counts the number of $k$-colorings of~$G$
such that every color-induced subgraph is $H$-free. As an extreme example,
if $G=H$, then $P_{\eta_G}(G;k)=k^{n(G)}-k$, because the non-$G$-free
colorings are precisely those using only one color.

If $H=K_{m,1}$, then $P_m(G;k)=P_{\eta_H}(G;k)$
counts the $k$-colorings such that no vertex
belongs to $m$ or more monochromatic edges, or equivalently such that
no color-induced
subgraph has a vertex of degree $\geq m$.
We call this the \emph{degree-chromatic polynomial}; if $m=1$, then $P_1(G;k)$
is just the usual chromatic polynomial.  In general, two trees with the same
number of vertices need not have the same degree-chromatic polynomials
for all $m$ (though they do share the same chromatic polynomial).  For example, if
$\mathsf{Z}$ is the three-edge path on four vertices and $\mathsf{Y}=K_{3,1}$
is the three-edge star, then $P_2(\mathsf{Z};k)=k^4-2k^2+k$ and
$P_2(\mathsf{Y};k)=k^4-3k^2+2k$.

In an earlier version of this article, we had conjectured, based on
experimental evidence, that if $T$ is any tree on $n$
vertices and $m<n$, then
  $$P_m(T;k)=k^n-\sum_{v\in V(T)}\binom{d_T(v)}{m}k^{n-m}+(\text{lower
  order terms})$$
where $d_T(v)$ denotes the degree of vertex $v$.
This conjecture has since been proven combinatorially by Diego
Cifuentes~\cite{Cif}.
\end{exa}

\section{Tutte characters}

The \emph{Tutte polynomial} $T_G(x,y)$ is a powerful graph invariant
with many important properties (for a comprehensive survey, see~\cite{BryOx}).
It is defined in closed form by the formula
$$
T_G(x,y) ~=~ \sum_{A\subseteq E(G)} (x-1)^{\rk(G)-\rk(A)} (y-1)^{\nul(A)}
$$
where $\rk(A)$ is the graph rank of $A$, and $\nul(A)=|A|-\rk(A)$ (the \emph{nullity} of~$A$).
The Tutte polynomial is a universal
deletion-contraction invariant in the sense that every
graph invariant satisfying a deletion-contraction recurrence can be
obtained from $T_G(x,y)$ via a standard ``recipe'' \cite[p.~340]{Bollobas}.
In particular, $T_G(x,y)$ is multiplicative on connected components, so we can regard it as a character on the graph algebra:
$$\tau_{x,y}(G) = T_G(x,y).$$
We may regard $x,y$ either as indeterminates or as (typically integer-valued) parameters.
It is often more convenient to work with the \emph{rank-nullity polynomial}
\begin{equation} \label{R-to-T}
R_G(x,y) = \sum_{A\subseteq E} (x-1)^{\rk(A)} (y-1)^{\nul(A)} = (x-1)^{\rk(G)}T_G(x/(x-1),y)
\end{equation}
which carries the same information as $T_G(x,y)$, and is also multiplicative on connected components, hence is a character on $\GG$.
Note that $R_G(1,y)=1$, and that
\begin{equation} \label{T-to-R}
T_G(x,y) = (x-1)^{\rk(G)}R_G(x/(x-1),y).
\end{equation}

Let $\rho_{x,y}$ denote the function $G\mapsto R_G(x,y)$, viewed as a character of the graph algebra $\GG$.

For later use, we record the relationship between $\rho$ and $\tau$:
\begin{equation} \label{tau-and-rho}
\tau_{x,y} = (x-1)^{\rk(G)} \rho_{x/(x-1),y}, \qquad
\rho_{x,y} = (x-1)^{\rk(G)} \tau_{x/(x-1),y}.
\end{equation}
In particular,
\begin{equation} \label{tau-and-rho-special}
\tau_{2,y} = \rho_{2,y}\qandq
\tau_{0,y} = \widetilde{\rho_{0,y}}.
\end{equation}

\subsection{The main theorem on Tutte characters}
Let
$$P_{x,y}(G;k)=\rho_{x,y}^k(G)$$
be the image of $G$ under the CHA morphism
$(\GG,\rho_{x,y})\to\Ff(x,y)[k]$ of Proposition~\ref{to-binomial-alg}.
Thus $P_{x,y}(G;k)\in\Cc(x,y)[k]$.  The main theorem of this section
is that $P_{x,y}(G;k)$ is itself essentially an
evaluation of the Tutte polynomial.
\begin{thm}
\label{tutte-char}
We have
$$P_{x,y}(G;k)
= k^{c(G)} (x-1)^{\rk(G)} T_G\left(\frac{k+x-1}{x-1}, y\right).$$
\end{thm}

\begin{proof}
Since $P_{x,y}(G;k)$ is a polynomial in $k$, it suffices to prove
that the identity holds for all positive integer values of~$k$.

We have
\begin{subequations}
\begin{align}
P_{x,y}(G;k) &= \rho_{x,y}^k(G) = \sum_{V_1\uplus\cdots\uplus V_k=V(G)}\quad \prod_{i=1}^k \rho_{x,y}(G|_{V_i}) \label{Pxy-sum}\\
&= \sum_{V_1\uplus\cdots\uplus V_k=V(G)}\quad \prod_{i=1}^k\quad \sum_{A_i\subseteq E(G|_{V_i})} (x-1)^{\rk(A_i)}(y-1)^{\nul(A_i)}\\
&= \sum_{f:V\to[k]}\quad \prod_{i=1}^k\quad \sum_{A_i\subseteq f^{-1}(i)} (x-1)^{\rk(A_i)}(y-1)^{\nul(A_i)}\\
&= \sum_{f:V\to[k]}\quad \sum_{A\subseteq M(f,G)} (x-1)^{\rk(A)}(y-1)^{\nul(A)}
\end{align}
\end{subequations}
where $M(f,G)$ denotes the set of $f$-monochromatic edges,
that is, edges $e=uv\in E(G)$ such that $f(u)=f(v)$ (including, in particular, all loops).
Here the sum is over all ordered partitions of $V(G)$ into pairwise disjoint subsets (possibly empty).
In order to find a recipe for $P_{x,y}(G;k)$ as a Tutte specialization, we need to know its value on edgeless graphs,
and how it behaves with respect to deleting a loop, deleting a cut-edge, or deletion and contraction of an ``ordinary'' edge.

\textit{Step 1: Edgeless graphs.} If $E(G)=\0$, then $R_H(x,y)=1$ for every subgraph $H\subseteq G$.  Therefore, every summand in \eqref{Pxy-sum} is 1, so $P_{x,y}(G;k)$ is just the number of ordered partitions with $n=|V(G)|$ parts, that is:
\begin{equation}
\label{recipe-alpha}
P_{x,y}(\overline{K_n};k)=k^n.
\end{equation}

\textit{Step 2: Loops.} Suppose $G$ has a loop $\ell$.  For every ordered partition $V(G)=V_1\uplus\cdots\uplus V_k$, let $V_i$ be the part that contains the endpoint of~$\ell$.  Then
$\rho_{x,y}(G|_{V_i}) = y\rho_{x,y}((G-\ell)|_{V_i})$, and we conclude that
\begin{equation}
\label{recipe-y}
P_{x,y}(G;k)=y\cdot P_{x,y}(G-\ell;k).
\end{equation}

\textit{Step 3: Nonloop edges.}
Suppose $G$ has a nonloop edge $e$ (possibly a cut-edge) with endpoints $u,v$.  For a function $f:V\to[k]$, if $f(u)\neq f(v)$ then $M(f,G-e)=M(f,G)$,
while if $f(u)=f(v)$ then $M(f,G-e)=M(f,G)\sm\{e\}$.  For every edge set $A\subseteq M(f,G)$ containing $e$, the edge set
$B=A\sm\{e\}\subseteq M(f,G/e)$ satisfies
$\nul(B)=\nul(A)$ and $\rk(B)=\rk(A)-1$; moreover,
the correspondence between $A$ and $B$ is a bijection. Therefore,
\begin{align*}
P_{x,y}(G;k) &= \sum_{f:V\to[k]} \sum_{A\subseteq M(f,G)} (x-1)^{\rk(A)}(y-1)^{\nul(A)},\\
P_{x,y}(G-e;k) &= \sum_{f:V\to[k]} \sum_{A\subseteq M(f,G-e)} (x-1)^{\rk(A)}(y-1)^{\nul(A)},\\
P_{x,y}(G;k)-P_{x,y}(G-e;k) &= \sum_{\substack{f:V\to[k]:\\e\in M(f,G)}} \sum_{\substack{A\subseteq M(f,G):\\e \in A}} (x-1)^{\rk(A)}(y-1)^{\nul(A)}\\
&= \sum_{\substack{f:V\to[k]:\\f(u)=f(v)}} \sum_{B\subseteq M(f,G/e)} (x-1)^{\rk(B)+1}(y-1)^{\nul(B)}\\
&= (x-1) P_{x,y}(G/e;k).
\end{align*}
To put this recurrence in a more familiar form,
\begin{equation}
\label{del-contr}
P_{x,y}(G;k)=P_{x,y}(G-e;k)+(x-1)P_{x,y}(G/e;k).
\end{equation}

\textit{Step 4: Cut-edges.} Now suppose that $e=uv$ is a cut-edge.  We have
$$P_{x,y}(G-e;k) = \sum_{f:V\to[k]} \sum_{A\subseteq M(f,G-e)} (x-1)^{\rk(A)}(y-1)^{\nul(A)}$$
and
$$
P_{x,y}(G/e;k) = \sum_{f:V\to[k]} \sum_{A\subseteq M(f,G/e)} (x-1)^{\rk(A)}(y-1)^{\nul(A)}.
$$
Let $H$ be the connected component of $G-e$ containing $u$, and let
$H'=G-H$.  Then $E(G-e)=E(H)\cup E(H')$.  Let the cyclic group
$\Zz_k$ act on colorings $f$ by cycling the colors of vertices in $H$
modulo~$k$ and fixing the colors of vertices in $H'$; i.e., if we
fix a generator $\gamma$ of $\Zz_k$, then $(\gamma^jf)(w)= f(w)+j\pmod{k}$
for $w\in V(H)$, while $(\gamma f)(w) = f(w)$ for $w\in V(H')$.  Then
the set $M(f,G-e)$ is invariant under the action of $\Zz_k$; moreover,
each orbit has size~$k$ and has exactly one coloring for which
$f(u)=f(v)$.  In that case, contracting the edge $e$ does not change
the nullity or rank.  Therefore, $P_{x,y}(G/e;k)=k^{-1}P_{x,y}(G-e;k)$,
which when combined with~\eqref{del-contr} yields
\begin{align}
P_{x,y}(G;k) &~=~ P_{x,y}(G-e;k)+(x-1) P_{x,y}(G-e;k)/k\notag\\
&~=~ \left(\frac{k+x-1}{k}\right) P_{x,y}(G-e;k).\label{recipe-x}
\end{align}  

Now combining \eqref{recipe-alpha}, \eqref{recipe-y}, \eqref{del-contr}, and \eqref{recipe-x} with the ``recipe theorem'' \cite[p.~340]{Bollobas} (replacing Bollob\'as' $x,y,\alpha,\sigma,\tau$ with $(k+x-1)/k$, $y$, $k$, $1$, $x-1$ respectively)
gives the desired result.
\end{proof}

\subsection{Applications to Tutte polynomial evaluations}
Theorem~\ref{tutte-char} has many enumerative consequences,
some familiar and some less so.  Many of the formulas we
obtain resemble those in the work of Ardila~\cite{Ardila};
the precise connections remain to be investigated.

First, observe that setting $x=y=t$ in Theorem~\ref{tutte-char} yields
\begin{align}
\rho_{t,t}^k(G)
~&=~ P_{t,t}(G;k)
 ~=~ k^{c(G)} (t-1)^{\rk(G)} T_G\left(\frac{k+t-1}{t-1}, t\right)\notag\\
~&=~ k^{c(G)} \bar\chi_G(k;t) \label{crapo}
\end{align}
where $\bar\chi$ denotes Crapo's coboundary polynomial; see
\cite[p.~236]{Cyclo} and \cite[\S6.3.F]{BryOx}.  (As a note, the bar in
the notation $\bar\chi$ has no relation to the bar involution
on $\Xx(\GG)$.)

\begin{cor}
For $k\in\Zz$ and $y$ arbitrary, the Tutte characters $\tau_{2,y}$ and $\tau_{0,y}$ satisfy the identities
\begin{align}
\left(\tau_{2,y}\right)^k(G) &= k^{c(G)} T_G(k+1, y), \label{two-formula}\\
\left(\widetilde{\tau_{0,y}
}\right)^k(G) &= k^{c(G)} (-1)^{\rk(G)} T_G(1-k,y). \label{zero-formula}
\end{align}
In particular,
$(\widetilde{\tau_{0,y}})^{-1} = \overline{\tau_{2,y}}.$
\end{cor}

\begin{proof}
Setting $x=2$ or $x=0$ in Theorem~\ref{tutte-char} and applying
\eqref{tau-and-rho-special}, we obtain respectively
\begin{align*}
\left(\tau_{2,y}\right)^k(G) &= \left(\rho_{2,y}\right)^k(G) = P_{2,y}(G;k) = k^{c(G)} T_G(k+1, y),\\
\left(\widetilde{\tau_{0,y}}\right)^k(G) &= \left(\rho_{0,y}\right)^k(G) = P_{0,y}(G;k) = k^{c(G)} (-1)^{\rk(G)} T_G(1-k,y).
\end{align*}
establishing \eqref{two-formula} and \eqref{zero-formula}.
In particular, setting $k=-1$ in \eqref{zero-formula} gives
$$(\widetilde{\tau_{0,y}})^{-1}(G) = (-1)^{c(G)} (-1)^{\rk(G)} T_G(2,y) = (-1)^{n(G)}\tau_{2,y}(G) = \overline{\tau_{2,y}}(G).\qedhere$$
\end{proof}

Similarly, we can find combinatorial interpretations of convolution
powers of the characters $\tau_{2,2}$, $\tau_{2,0}$,
$\widetilde{\tau_{0,2}}$, and $\widetilde{\tau_{0,0}}$.  In the last
case, we recover the standard formula for the chromatic polynomial as
a specialization of the Tutte polynomial (note that
$\widetilde{\tau_{0,0}}=\tau_{0,0}=\zeta$).

One can deduce combinatorial interpretations of other evaluations
of the Tutte polynomial.
If $G$ is connected, then substituting $y=2$ and $k=2$ into \eqref{two-formula} yields $(\tau_{2,2})^2(G)=2T_G(3,2)$, or
\begin{equation}
\label{T32}
T(G;3,2)
 ~=~ \frac{(\tau_{2,2}*\tau_{2,2})(G)}{2}
~=~ \sum_{U\subseteq V(G)} 2^{e(G|_U)+e(G|_{\bar U})-1}.
\end{equation}
That is, $2T(G;3,2)$ counts the pairs $(f,A)$, where $f$ is
a 2-coloring of~$G$ and $A$ is a set of $f$-monochromatic edges.

In order to interpret more general powers of Tutte characters, we use
\eqref{tau-and-rho} to expand the convolution power $\rho_{x,y}^k(G)$ in
Theorem~\ref{tutte-char}:
\begin{align*}
k^{c(G)} (x-1)^{\rk(G)} &T_G\left(\frac{k+x-1}{x-1}, y\right)
~=~ \rho^k_{x,y}(G)\\
~&=~ \sum_{V_1 \uplus \dots \uplus V_k = V(G)} \ \ \prod_{i=1}^k \rho_{x,y}(G_i)\\
~&=~ \sum_{V_1 \uplus \dots \uplus V_k = V(G)} \ \ \prod_{i=1}^k (x-1)^{\rk(G_i)} \tau_{x/(x-1),y}(G_i)
\end{align*}
where $G_i=G|_{V_i}$.  Note that in the special case $G=K_n$, we have $G_i\cong K_{|V_i|}$ and $\rk(G_i)=|V_i|-1$
for all $i$, so the equation simplifies to
$$k(x-1)^{n-1}T_{K_n}\left(\frac{k+x-1}{x-1}, y\right)=(x-1)^{n-k}\tau_{x/(x-1),y}(K_n)$$
or
\begin{equation} \label{tau-Kn-power}
(x-1)^{k-1} T_{K_n}\left(\frac{k+x-1}{x-1}, y\right)
= k^{-1} (\tau_{x/(x-1),y})^k(K_n).
\end{equation}
This equation has further enumerative consequences:
setting $x=2$ gives
\begin{align} \label{Kn-x2}
T_{K_n}(k+1,y)
&= \frac{1}{k} \sum_{a_1+\cdots+a_k=n} \frac{n!}{a_1!a_2!\dots a_k!} \tau_{2,y}(K_{a_1}) \dots \tau_{2,y}(K_{a_k}).
\end{align}
Setting $y=0$ in \eqref{Kn-x2} and observing that $\tau_{2,0}(K_a)=a!$ (the number
of acyclic orientations of $K_a$), we get
$T_{K_n}(k+1,0)=(n+k-1)!/k!$.  This is not a new formula; it follows from
the standard specialization of the Tutte polynomial
to the chromatic polynomial~\cite[Prop.~6.3.1]{BryOx},
together with the well-known formula
$k(k-1)\cdots(k-n+1)$ for the chromatic polynomial of~$K_n$.
On the other hand, setting $y=2$ in~\eqref{Kn-x2}, and recalling that $\tau_{2,2}(K_a)=2^{|E(K_a)|}=2^{\binom{a}{2}}$, gives
\begin{equation} \label{Kn-x2-y2}
T_{K_n}(k+1,2)
~=~ k^{-1}\!\!\! \sum_{a_1+\cdots+a_k=n} \frac{n!}{a_1!a_2!\dots a_k!} 2_{\phantom{X}}^{\binom{a_1}{2}+\cdots+\binom{a_k}{2}}.
\end{equation}
This formula may be obtainable from the generating function
for the Crapo coboundary polynomials of complete graphs, as computed
by Ardila~\cite[Thm.~4.1]{Ardila};
see also sequence~\href{http://oeis.org/A143543}{A143543} in \cite{EIS}.  Notice that
setting $k=2$ in~\eqref{Kn-x2-y2} recovers~\eqref{T32}.

It is natural to ask what happens when we set $x=1$,
since this specialization of the Tutte polynomial
has well-known combinatorial interpretations
in terms of, e.g., the chip-firing game~\cite{Merino}
and parking functions \cite{GS}.
The equations \eqref{R-to-T} and \eqref{T-to-R} degenerate
upon direct substitution, but we can instead 
take the limit of both sides of Theorem~\ref{tutte-char} as $x\to1$,
obtaining (after some calculation, which we omit)
$$\rho_{1,y}^k(G)= k^{n(G)}.$$

What can be said about Tutte characters in light of
Proposition~\ref{to-binomial-alg}?
Replacing $x$ with $(k+x-1)/(x-1)$ in Theorem~\ref{tutte-char}, we get
\begin{equation}
\label{jimmy}
\begin{aligned}
P_{(k+x-1)/(x-1),y}(G;k) ~&=~ k^{c(G)} (k/(x-1))^{\rk(G)} T(G;x,y)\\
~&=~ k^{n(G)} (x-1)^{-\rk(G)} T(G;x,y).
\end{aligned}
\end{equation}
One consequence is a formula for the Tutte polynomial
in terms of $P$:
\begin{equation}
\label{Txy-form-new}
T(G;x,y) = k^{-n(G)} (x-1)^{\rk(G)} P_{(k+x-1)/(x-1),y}(G;k).
\end{equation}
In addition, the left-hand-side of~\eqref{jimmy}
--- which is an element of $\Ff(x,y)[k]$ --- is actually just $k^{n(G)}$ times
a rational function in $x$ and $y$.  Setting $k=x-1$ or $k=1-x$,
we can write down simpler formulas for the Tutte polynomial in terms of $P$:
\begin{align*}
T(G;x,y) &= (x-1)^{-c(G)} P_{2,y}(G;x-1),\\
T(G;x,y) &= (-1)^{n(G)} (x-1)^{c(G)} P_{0,y}(G;1-x).
\end{align*}

\section{A reciprocity relation between $K_n$ and $K_m$}
For each nonzero scalar~$c\in\Ff$, there is a character~$\xi_c$ on~$\GG$
defined by $\xi_c(G)=c^{n(G)}$.  In this concluding section, we list some
basic identities involving convolution powers of these characters
and their interactions with the character $\zeta$.  The main result,
Proposition~\ref{wacko-reciprocity}, is
a ``reciprocity'' relation between the complete graphs $K_n$ and $K_m$.

First, let $c,d\in\Ff$ be arbitrary nonzero scalars, and let $k$ be an integer.
The definition of convolution, together with a straightforward
application of the binomial theorem, yields
the identities
$$\xi_c*\xi_d=\xi_{c+d}, \qquad\xi_c^k=\xi_{ck}, \qquad
\xi_c^{-1}=\xi_{-c}=\overline{\xi_c}.$$
In particular, the characters $\xi_c$ form a subgroup of $\Xx(\GG)$ isomorphic
to the additive group $\Ff$.
Another easily obtained fact is the following: for every graph $G$,
$$(\zeta*\xi_c)(G)=\sum_{\text{cocliques } Q} c^{|V(G)|-|Q|}.$$

\begin{prop}  \label{wacko-reciprocity}
For all $n,m\in\Zz_{\geq0}$, we have
$$(\zeta^n*\xi_1)(K_m) = (\zeta^m*\xi_1)(K_n).$$
\end{prop}

\begin{proof}
Consider the action of the character
$\zeta^n*\xi_1$ on the graph $K_m$:
\begin{align}
(\zeta^n*\xi_1)(K_m)
&= \sum_{V\subseteq[m]} \zeta^n(K_m|_V) \xi_1(K_m-V)\notag\\
&= \sum_{j\in\Zz} \binom{m}{j} \zeta^n(K_j) \xi_1(K_{m-j})\notag\\
&= \sum_{j\in\Zz} \binom{m}{j} \binom{n}{j} j!. \label{is-symm}
\end{align}
Here the summand $j$ corresponds to $m-|V|$; note that
$\zeta^n(K_j)=\binom{n}{j} j!$ is the number of $n$-colorings of
$K_j$.  (We are using the convention that $\binom{n}{j}$ vanishes
when $j<0$ or $j>n$.)  The expression \eqref{is-symm} is symmetric
in $m$ and $n$, implying the desired result.
\end{proof}

\begin{cor} \label{funky}
Let $m,n$ be nonnegative integers.  Then
\[
(\zeta^n*\xi_{-1})(K_m) = (-1)^{n+m}(\zeta^m*\xi_{-1})(K_n).
\]
\end{cor}

\begin{proof}
The desired identity can be obtained by applying the bar involution to both sides of Proposition~\ref{wacko-reciprocity}
(or, equivalently, redoing the calculation, replacing~$\xi_1$ with~$\xi_{-1}$ throughout).
\end{proof}

Experimental evidence indicates that
$$(\zeta^{-1}*\xi_1)(K_n) = (-1)^n  D_n, \qquad (\zeta^{-1}*\xi_{-1})(K_n) = (-1)^n  A_n,$$
where $D_n$ is the number of derangements of $\{1,2,\dots,n\}$ and $A_n$ is the number of arrangements (sequences \href{http://oeis.org/A000166}{A000166} and \href{http://oeis.org/A000522}{A000522} of \cite{EIS}, respectively).
More generally, we have conjectured that for every scalar~$c$ and integer~$k$,
the
exponential generating function for $(\zeta^k*\xi_c)(K_n)$ is
\begin{equation} \label{EGF}
\sum_{n\geq 0} (\zeta^k*\xi_c)(K_n) \frac{x^n}{n!} ~=~ e^{cx}(1+x)^k
\end{equation}
(see \cite[Example~2.2.1]{EC1}, \cite[Example~5.1.2]{EC2}).
In fact, formula~\eqref{EGF} follows from independent, unpublished work of
Aguiar and Ardila~\cite{AA-PC}.

\bibliographystyle{amsalpha}
\bibliography{biblio}
\end{document}